\documentclass{amsart}
\usepackage{amsmath,amssymb,stmaryrd,amsthm}
\usepackage{graphicx}
\usepackage{accents}
\usepackage[color=yellow]{todonotes}

\newcommand{\jump}[1]{[\![#1]\!]}

\newcommand{\pp}[2]{\frac{\partial #1}{\partial #2}} 
\newcommand{\dede}[2]{\frac{\delta #1}{\delta #2}}
\newcommand{\dd}[2]{\frac{\diff#1}{\diff#2}}

\DeclareMathOperator{\diff}{d}
\DeclareMathOperator{\sgn}{sgn}

\newtheorem{theorem}{Theorem}
\newtheorem{corollary}[theorem]{Corollary}
\newtheorem{definition}[theorem]{Definition}
\newtheorem{remark}[theorem]{Remark}
\newtheorem{proposition}[theorem]{Proposition}

\newcommand{\response}[1]{{{#1}}}
\usepackage{natbib}

\usepackage{color}
\usepackage[margin=2cm]{geometry}

\usepackage{fancybox}
\usepackage{xspace}
\usepackage[colorlinks,citecolor=cyan]{hyperref}
\usepackage{xcolor}

\newcommand{\HH}{\ensuremath{\operatorname H}\xspace}

\newcommand{\WW}{\ensuremath{\operatorname W}\xspace}

\newcommand{\sob}[2]{\ensuremath{\WW^{#1,#2}}}

\newcommand{\sobh}[1]{\ensuremath{\HH^{#1}}}

\newcommand{\norm}[1]{\ensuremath{\left|#1\right|}}

\newcommand{\qp}[1]{\ensuremath{\!\left({#1}\right)}}

\newcommand{\naturals}{\ensuremath{\mathbb N}\xspace}

\renewcommand*{\dot}[1]{\accentset{\mbox{\large\bfseries .}}{#1}}

\title{Singular solutions of the r-Camassa-Holm equation}
\author{
  Colin J. Cotter
}
\address{
  Colin J. Cotter
  \thanks{
    Department of Mathematics, Imperial College London, UK.
    {\tt{colin.cotter@imperial.ac.uk}}.
}}

\author{
  Darryl D. Holm 
}
\address{
  Darryl D. Holm
  \thanks{
    Department of Mathematics, Imperial College London, UK.
    {\tt{d.holm@imperial.ac.uk}}.
}}

\author{
  Tristan Pryer
}
\address{
  Tristan Pryer
  \thanks{
    Department of Mathematical Sciences, University of Bath, Bath BA2 7AY, UK.
    {\tt{tmp38@bath.ac.uk}}.
}}

\begin{document}

\maketitle

\begin{abstract}
  This paper introduces the r-Camassa-Holm (r-CH) equation, which
  describes a geodesic flow on the manifold of diffeomorphisms acting
  on the real line induced by the $\sob{1}{r}$ metric. \response{The conserved
  energy for the problem is given by the full $\sob{1}{r}$ norm. For $r=2$, we recover the Camassa-Holm equation.} We compute
  the Lie symmetries for r-CH and study various symmetry reductions. 
  We introduce singular weak solutions of the r-CH equation for $r\geq
  2$ and demonstrates their robustness in numerical simulations of
  their nonlinear interactions in both overtaking and head-on
  collisions. Several open questions are formulated about the
  unexplored properties of the r-CH weak singular solutions, including
  the question of whether they would emerge from smooth initial
  conditions.
\end{abstract}

\section{Introduction}

The Camassa-Holm (CH) equation is a 1+1 dimensional, variational,
partial differential equation (PDE). It was originally derived as
a shallow water model for the propagation of water waves that
incorporates non-hydrostatic effects \citep{camassa1993integrable}.
The general form of the CH equation is
\begin{equation}
  u_t + 2\kappa u_x - u_{xxt} + 3uu_x = 2u_xu_{xx} + uu_{xxx},
  \label{CH-eqn}
\end{equation}
which is considered here on the real line with boundary conditions
$u,u_x\to 0$ as $|x|\to \infty$. This paper considers the case $\kappa=0$,
for which the CH equation has two properties of particular interest.  

The first property of interest here is that the CH equation \eqref{CH-eqn} is the
Euler-Poincar\'e equation \citep{holm1998euler} for Hamilton's principle
$\delta S = 0$ with $S = \int l[u]\,dt$ for Lagrangian $l[u]$ given by
\begin{equation}
  \label{eq:ch energy}
  l[u] = \frac{1}{2}\int_{-\infty}^{\infty} u^2 + u_x^2 \diff x.
\end{equation}
The CH equation describes a geodesic flow on the diffeomorphism
group \citep{kouranbaeva1999camassa}, with respect to the kinetic energy metric
induced by the $\sobh{1}$ norm $l[u]=\|u\|^2_{\sobh{1}}$. 
Hence, solutions of the CH equation \eqref{CH-eqn} conserve the kinetic energy $l[u]$. In
fact, the CH equation has an infinite number of conserved quantities, because it is
bi-Hamiltonian. That is, CH has two compatible Poisson structures. 
The bi-Hamiltonian property of CH implies its 
isospectrality, which in turn implies its integrability as a Hamiltonian system. 

The second property of interest here is that the CH equation admits weak solutions
given by the singular momentum map of \cite{holm2005momentum},
\begin{equation}
  \label{eq:mom map}
  u(x,t) - u_{xx} = \sum_{i=1}^N P_i(t)\delta(x-Q_i(t)),
\end{equation}
or equivalently,
\begin{equation}
  \label{eq:peakons}
  u(x,t) = \sum_{i=1}^N P_i(t)\exp(-|x-Q_i(t)|).
\end{equation}
These are the ``$N$-peakon solutions'' of the CH equation. The quantities
$\{(P_i,Q_i)\}_{i=1}^N$ in these solutions are canonically conjugate
variables in a Hamiltonian system, for the Hamiltonian obtained by
inserting the peakon Ansatz \eqref{eq:peakons} into the energy
\eqref{eq:ch energy}.

The $N$-peakon solutions of the CH equation provide a useful (in fact,
canonical) example of blow-up of geodesics, since a peakon and an
anti-peakon can collide in finite time (as described in
\citet{camassa1993integrable}), at which point the solution loses
uniqueness in $\sobh{1}$. This is possible because the global proof of
existence and uniqueness of geodesics found, e.g., in
\citet{younes2010shapes} requires the norm defined by $l[u]$ to be at
least as strong as $\sob{1}{\infty}$.

In this paper, we change $l[u]$ from the $\sobh{1}$ \response{norm} to a scaled
$\sob{1}{r}$ \response{norm}, specifically
\begin{equation}
  \label{eq:rch energy}
  l[u] = \frac{1}{r}\int_{-\infty}^{\infty} |u|^r + \frac{1}{r-1}|u_x|^r \diff x,
\end{equation}
and we consider the corresponding Euler-Poincar\'e equation. Note the
scaling of $1/(r-1)$ is included here so that the formula
\eqref{eq:peakons} still holds for $N=1$, as we shall see later.

\begin{definition}[r-Camassa-Holm (r-CH) equation]
  The r-Camassa-Holm equation is
  \begin{align}
    \label{eq:rCH}
    \left(|u|^{r-2}u - \frac{1}{r-1}(|u_x|^{r-2}u_{\response{x}})_x\right)_t +
     \left[\left(|u|^{r-2}u -  \frac{1}{r-1}\left(|u_x|^{r-2}u_x\right)_x\right)u\right]_x & \\
\quad     +
    \left(|u|^{r-2}u - \frac{1}{r-1}(|u_x|^{r-2}u_{\response{x}})_x\right)u_x & = 0, \nonumber
  \end{align}
  with boundary conditions $u\to 0$ as $x\to \pm\infty$
  (or periodic boundary conditions).
\end{definition}
The r-CH equation in \eqref{eq:rCH} describes a geodesic flow
on the manifold of diffeomorphisms acting on the real line induced by the $\sob{1}{r}$
metric. The conserved energy is $\|u\|_{\sob{1}{r}}^r$.  The
$r\to\infty$ limit is particularly interesting because the conserved energy
approaches the $\sob 1 \infty$ norm, which is the minimum regularity
required for global existence of geodesics, as mentioned above. This
limit has been explored in \citep{bauer2021can}.

\response{When $r=2$, \eqref{eq:rCH} becomes
  \begin{equation}
    \label{eq:2CH}
    \left(u - u_{xx}\right)_t +
    \left(u - u_{xx}\right)_xu +
    2\left(u - u_{xx}\right)u_x = 0,
  \end{equation}
which we recognise as \eqref{CH-eqn} with $\kappa=0$.}

One may now ask whether the r-CH equation \eqref{eq:rCH} also has weak singular
solutions, akin to the peakons for the CH equation. Formally, the
answer is yes, because the Holm-Marsden singular momentum map exists
for the Euler-Poincar\'e equation on the diffeomorphism group for any
given Lagrangian \citep{holm2005momentum}. This leads to further
questions, though.  In what sense do these singular solutions solve
the r-CH equation \eqref{eq:rCH}? Can these solutions be constructed explicitly?
Can their behaviour be analysed enough to be computed? These
questions will be settled in this paper.

This paper constructs singular solutions that conserve the $\sob{1}{r}$
norm and proves that they are solutions of the r-CH equation \eqref{eq:rCH}. In
particular, they are weak solutions of the spatial integral of the r-CH
equation. This paper builds upon techniques developed for the
r-Hunter-Saxton equation in \citet{cotter2020r}, but here the
challenge is that the singular solutions are no longer piecewise linear,
which substantially complicates things. Remarkably, these nonlinear complications
do not prevent the dynamical equations for the point locations and
their canonical momenta from having a very simple form.

The rest of this paper is structured as follows.  In section
\ref{sec:r-CH}, we derive the r-CH equation as an Euler-Poincar\'e
equation, and we discuss its integrated form and its weak solutions.  In
section \ref{sec:singular}, we describe the construction of singular
solutions from a Hamiltonian system for the peak locations and their
corresponding canonical momenta. We then prove that these 
singular solutions are indeed weak solutions of the integrated form of
the r-CH equation.  In section \ref{sec:solns}, we discuss singular solutions
for $N=1,2,3$, presenting numerical simulations
in the latter two cases. Finally, in section \ref{sec:summary}, we provide
a summary and outlook.

\section{The r-CH equation}
\label{sec:r-CH}

In this section we formally introduce the r-CH equation and describe some
of its properties.

\begin{proposition}
  The r-CH equation \eqref{eq:rCH} is the Euler-Poincar\'e equation
  for the Lagrangian
  \begin{equation}
    l_r[u] = \frac{1}{r}\int_{-\infty}^{\infty} |u|^r + \frac{1}{r-1}|u_x|^r \diff x.
  \label{rCH-Lag}
  \end{equation}
\end{proposition}
\begin{proof}
  One may derive the equations formally (assuming smooth solutions) using
  Hamilton's principle,
  \begin{equation}
    \delta S[u] = \delta \int_0^T l[u] \diff t = 0,
  \end{equation}
  with endpoint conditions $u(x,0)=u_0(x)$, $u(x,T)=u_T(x)$ and
  boundary conditions $u\to 0$ as $x\to \pm \infty$. For an arbitrary
  smooth vector field $w \in \sob{1}{r}$ with $w(x,0)=w(x,T)=0$ and
  $w\to 0$ as $x\to \pm\infty$ (see \citet{holm1998euler}), one finds,
  \begin{align}
  \begin{split}
    0 = \delta S[u] &= \delta \int_0^T l[u] \diff t = 0, \\
    & = \int_0^T\int_{-\infty}^{\infty} |u|^{r-2}u\delta u + \frac{1}{r-1}|u_x|^{r-2}u_x\delta u_x
    \diff x\diff t.
  \end{split}
  \end{align}
  \newpage
  
  Upon integrating by parts we have
  \begin{align}
    \begin{split}
  0 & = \int_0^T\int_{-\infty}^{\infty} \left(|u|^{r-2}u -  \frac{1}{r-1}\left(|u_x|^{r-2}u_x\right)_x\right)\delta u
    \diff x\diff t
    \response{=\int_0^T\int_{-\infty}^{\infty} m \,\delta u\,\diff x\diff t}, \\
    \label{eq:weak rCH}
    & = \int_0^T\int_{-\infty}^{\infty} \left(|u|^{r-2}u -  \frac{1}{r-1}\left(|u_x|^{r-2}u_x\right)_x\right)\left(\dot{w} - wu_x + uw_x\right)
    \diff x\diff t. \\
    \end{split}
  \end{align}
  By using the \response{Euler-Poincar\'e} constrained variations $\delta u = \dot{w} - wu_x +
  uw_x$ \response{we find}
  \begin{align}
    \begin{split}
      0 & = -\int_0^T\int_{-\infty}^{\infty} \Bigg(
      \left(|u|^{r-2}u -  \frac{1}{r-1}\left(|u_x|^{r-2}u_x\right)_x\right)_t
    + \left(|u|^{r-2}u -  \frac{1}{r-1}\left(|u_x|^{r-2}u_x\right)_x\right)u_x  \\
    & \qquad \qquad \qquad \qquad
    + \left(\left(|u|^{r-2}u -  \frac{1}{r-1}\left(|u_x|^{r-2}u_x\right)_x\right)u\right)_x
    \Bigg)w
    \diff x\diff t,
    \end{split}
  \end{align}
  from which the r-CH equation \eqref{eq:rCH} follows, since $w$ is arbitrary.
\end{proof}
\begin{corollary}
  \label{cor:lp}
  The r-CH equation is Lie-Poisson, with Hamiltonian $H:W^{1,r'}\mapsto\mathbb{R}$
  \begin{equation}
    H[m] = \frac{r-1}{r}\int_{\Omega} |u|^r + \frac{1}{r-1} |u_x|^r\diff x,
  \label{rCH-Ham}
  \end{equation}
  {where $u \in \sob{1}{r}$ is defined from $m \in
    \sob{1}{r'}$ \response{in the weak ($L^2$ integral) sense that} 
    \begin{equation}
      \label{eq:mom map m}
      \int_{-\infty}^{\infty} |u|^{r-2}u v + \frac{1}{r-1}|u_x|^{r-2}\response{u_x}v_x \diff x
      \response{=:}
      \int_{-\infty}^{\infty} m v \diff x
      \quad \forall v \in \sob{1}{r},
    \end{equation}
  with
  \begin{equation}
    \frac 1 r + \frac 1 {r'} = 1.
  \end{equation}
  } 
  The Lie-Poisson bracket $\{\cdot,\cdot\}:(W^{1,r'})'\times(W^{1,r'})'\mapsto \mathbb{R}$ (where $W^{1,r'}$ is the space of bounded linear functionals on $W^{1,r'}$) is given by
  \begin{equation}
    \{F,G\} = \int_{-\infty}^{\infty} \left[\dede{F}{m},\dede{G}{m}\right]m \diff x,
  \label{LP-brkt}
  \end{equation}
  where $[u,v]=u_xv-v_xu$ is the usual Lie bracket for vector fields.
\end{corollary}

\begin{remark}
  The Lie-Poisson bracket in \eqref{LP-brkt} emerges from the momentum
  map obtained via the right action of the manifold of diffeomorphisms
  on its cotangent bundle (particle relabelling). The right action is
  a symmetry of the Lagrangian \eqref{rCH-Lag}. Hence, the right
  action leads to reduction of the cotangent bundle of the
  diffeomorphisms to the dual 
  \response{(with respect to the $L^2$ pairing)} of their Lie algebra of smooth vector fields (the 1-form densities).  In contrast, the singular solutions
  in \eqref{eq:mom map} represent the momentum map obtained via the
  corresponding left action of the diffeomorphisms on their cotangent
  bundle (particle motion).  However, the left action is not a
  symmetry of the Lagrangian \eqref{rCH-Lag} and, hence, the left
  action does not lead to a reduction of the cotangent bundle of the
  diffeomorphisms to a smaller space. \response{Instead, the momentum map
  for the left action yields another class of singular solutions, comprising the 
  left action of the diffeomorphisms on embedded subspaces of the flow domain.} 
  The two momentum maps are \response{weakly symplectically orthogonal and each 
  of them is} infinitesimally equivariant (\cite{holm2005momentum}). Together,
  they comprise a weak dual pair (\cite{gay2012dual}).
\end{remark}
\begin{proof}{of Corollary \ref{cor:lp}.}
  Note that formula \eqref{eq:mom map m} is solvable for $u$ from $m$
  by standard energy methods.  Hence, one may compute the Legendre
  transform by defining $u$ as above and substituting into formula
  \eqref{eq:mom map m}, to find
  \begin{align}
    \begin{split}
    H[m] & = \int_{-\infty}^{\infty} um \diff x - l[u], \\
    & = \int_{-\infty}^{\infty} |u|^{r-2}u^2 + \frac{1}{r-1}
    |u_x|^{r-2}(u_x)^2 \diff x - \frac{1}{r}\int_{-\infty}^{\infty} |u|^r + \frac{1}{r-1}|u_x|^r \diff x, \\
    & = \frac{r-1}{r}\int_{-\infty}^{\infty} |u|^r
    + \frac{1}{r-1}|u_x|^r \diff x.
    \end{split}
    \end{align}
  Then, the result follows from the \response{Euler-Poincar\'e} theorem of
  \cite{holm1998euler}. Alternatively, we can verify the Lie-Poisson
  structure by choosing $w=\dede{F}{m}$ in \eqref{eq:weak rCH}, to
  give
  \begin{equation}
     \dot{F} = \int_{-\infty}^{\infty}
     \dede{F}{m}\dot{m} \diff x 
     = \int_{\infty}^{\infty}m\left(\dede{F}{m}u_x - u\left(\dede{F}{m}\right)_x\right)\diff x.
  \end{equation}
  The results in the statement of Corollary \ref{cor:lp} will follow,
  provided we can show that $\dede{H}{m}=u$. We compute this relation
  by using the adjoint technique. That is, we define the functional
  $J:W^{1,r}\times W^{1,r'}\times W^{1,r}\mapsto \mathbb{R}$ by
  \begin{equation}
    J[u,m,v]
    =
    l[u] + \int_{\Omega} vm - |u|^{r-2}uv
    -
    \frac{1}{r-1}|u_x|^{r-2}u_xv_x\diff x.
  \end{equation}
  Upon setting $U(m)=u$ where $u$ satisfies \eqref{eq:mom map m}, we have,
  \begin{equation}
    H[m] = J[U(m),m,v],
  \end{equation}
  for any $v\in W^{1,r}$, which we are now free to choose in order to make the gradient
  computation easier. Then,
  \begin{align}
    \diff_m H[m; \delta m]
    &= \diff_u|_{u=U(m)}J[u, m, v; \dede{U(m)}{m}\delta m]
    + \diff_m|_{u=U(m)}J[u, m, v; \delta m], \\
    &=
    \diff_m J[U(m), m, v; \delta m] = \int_{-\infty}^{\infty} v\delta m \diff x,
    \quad \forall \delta m,
  \end{align}
  provided that $v\in\sob{1}{r}$ solves the adjoint equation
  \begin{align}
  0 &= \diff_u|_{u=U(m)}J[u, m, v; \response{\delta u}], \\
   & = (r-1)\int_{\Omega} |u|^{r-2}u \delta u + \frac{1}{r-1} \norm{u_{x}}^{r-2}u_x \delta u_x -
    \delta u \norm{u}^{r-2} v
    - \frac{1}{r-1}\delta u_x |u_x|^{r-2} v_x \diff x, \quad \forall \delta u \in \sob{1}{r},
  \end{align}
  which is well-posed for $r=2$, and also for $r>2$ provided that the
  set $\{x:u_x(x)=0\}$ has measure zero.  By substitution, we notice
  that $v=u$ is the solution. Thus, we have
  \begin{equation}
    \int_{-\infty}^{\infty} \Big(\dede{H}{m} - u\Big) \delta m
    \diff x
    = 0,
    \quad \forall \delta m \in \sob{1}{r'},
  \end{equation}
  as required.
\end{proof}

\begin{remark}[Singular solutions]
  In this paper we will consider peaked solutions, with singularities
  in the first derivative. These solutions are not regular enough for
  the strong form of r-CH \eqref{eq:rCH}, or even the weak form
  \eqref{eq:weak rCH}. To reconcile this, we must introduce the
  integrated form given below.
\end{remark}

\begin{proposition}[Integrated form of r-CH]
  Let $u$ solve \eqref{eq:rCH}\response{, with far field boundary conditions $u\to 0$ as $x\to \infty$.} Then, upon integration,
  \begin{equation}
    \label{eq:fi}
    \left(\psi - \frac{1}{r-1}|u_x|^{r-2}u_x\right)_t + \frac{1+r}{r}|u|^r
    - \frac{1}{r-1}(|u_x|^{r-2}u_xu)_x + \frac{1}{r(r-1)}|u_x|^r = 0,
  \end{equation}
 Here, $\psi(x)$ is defined as
  \begin{equation}
    \psi(x) = \int_{-\infty}^x |u|^{r-2}u(s) \diff s.
  \end{equation}
  Conversely, if $u$ satisfies \eqref{eq:fi} and $u$ is sufficiently
  smooth, then $u$ satisfies \eqref{eq:rCH}.
\end{proposition}
\begin{proof}
  Note that
  $$(|u_x|^{r-2}u_x)_x=(r-1)|u_x|^{r-2}u_{xx}$$
  so
  $$(\norm{u_x}^{r-2} u_x)_xu_x =  (r-1)|u_x|^{r-2}u_xu_{xx} = \frac{r-1}{r}(|u_x|^r)_x.$$
  Using this, together with $\psi_x=|u|^{r-2}u$, (\ref{eq:rCH}) can be written as 
  \begin{equation}
    \label{eq:chsimp}
    \begin{split}
      0 & =     \left(|u|^{r-2}u - \frac{1}{r-1}(|u_x|^{r-2}u_x)_x\right)_t +
      \left(\left(|u|^{r-2}u - \frac{1}{r-1}(|u_x|^{r-2}u_x)_x\right)u\right)_x
      \\
      &\qquad +
      \left(|u|^{r-2}u - \frac{1}{r-1}(|u_x|^{r-2}u_x)_x\right)u_x, \\
      & =  \left(\psi - \frac{1}{r-1}|u_x|^{r-2}u_x\right)_{tx}
      + \left(\left(|u|^{r-2}u - \frac{1}{r-1}(|u_x|^{r-2}u_x)_x\right)u\right)_x +
      \frac{1}{r}(|u|^r)_x - \frac{1}{r}(|u_x|^r)_x.
    \end{split}
  \end{equation}
 Integrating (\ref{eq:chsimp}) yields,
   \begin{align}
   \begin{split}
    f(t) &= \left(\psi - \frac{1}{r-1}|u_x|^{r-2}u_x\right)_{t}
    + \frac{r+1}{r}|u|^r - \frac{1}{r-1}\underbrace{(|u_x|^{r-2}u_x)_xu}_{=(|u_x|^{r-2}u_xu)_x - |u_x|^r} - \frac{1}{r}|u_x|^r, \\
    & = \left(\psi - \frac{1}{r-1}|u_x|^{r-2}u_x\right)_{t}
    + \frac{r+1}{r}|u|^r - \frac{1}{r-1}(|u_x|^{r-2}u_xu)_x + \frac{1}{(r-1)r}|u_x|^r, \\
 \end{split}
   \end{align}
   for some $x$-independent function $f$.  Noting the far field
   conditions, $u\to 0$ as $x\to \pm\infty$, we conclude that
   $f(t)=0$, as required.
\end{proof}

In this paper, we shall make use of the following weak form of
\eqref{eq:fi}.
\begin{definition}[Weak integrated form of r-CH]
  The function $u\in \sob{1}{r}$ satisfies the weak integrated form of
  the r-CH equation (\ref{eq:fi}), provided
  \begin{equation}
    \int_{-\infty}^{\infty}\phi\left(
    \left(\psi - \frac{1}{r-1}|u_x|^{r-2}u_x\right)_t + \frac{1+r}{r}|u|^r
    +
    \frac{1}{r(r-1)}|u_x|^r\right)
    +
    \frac{1}{r-1}(|u_x|^{r-2}u_xu)\phi_x \diff x=0,
\quad \forall \phi \in \sob{1}{r}.
  \end{equation}
\end{definition}
\response{Section \ref{sec: singsolns}} will show that equation \eqref{eq:fi} admits
weak solutions in $\sob{1}{r}$ and will also derive the corresponding singular
solutions.

\section{Symmetries and special solutions of the $1$-CH equation}
\label{sec:sym}

In this section we examine some symmetries and characterise some
special solutions for the $1$-CH equation.
\subsection*{Symmetries}
A Lie point symmetry of equation (\ref{eq:fi}) is a flow
\begin{equation}
  \label{eq:flow}
  \qp{\widetilde{x},\widetilde{t},\widetilde{u}}
  =
  \qp{e^{\epsilon X}x,e^{\epsilon X}t,e^{\epsilon X}u},
\end{equation}
generated by a vector field
\begin{equation}
  \label{eq:vf}
  X
    =
    \xi^1(x,t,u)\frac{\partial}{\partial x}
    +
    \xi^2(x,t,u)\frac{\partial}{\partial t}
    +
    \eta(x,t,u)\frac{\partial}{\partial u},
  \end{equation}
  such that $\widetilde{u}(\widetilde{x},\widetilde{y})$ is a solution
  of (\ref{eq:fi}) whenever $u(x,y)$ is a solution of
  (\ref{eq:fi}). As usual, we denote by $e^{\epsilon X}$ the
  \emph{Lie series} $\sum_{k=0}^{\infty}\frac{\epsilon^k}{k!}X^k$ with
  $X^{k}=XX^{k-1}$ and $X^0=1$.
  
  To find the symmetries of the $r$-CH equation we are required to
  solve the infinitesimal invariance condition for the vector field
  \eqref{eq:vf}. To do this we use the prolongation of $X$
  (\cite{Olver:1993}). The infinitesimal symmetry condition decomposes
  to a large overdetermined system of linear PDEs for $\xi^1$, $\xi^2$
  and $\eta$ known as \textit{determining equations}. The following
  three Propositions give the overdetermined system, the general form
  of the determining equations and the Lie algebra generators. These
  results were proven symbolically using the SYM package
  (\cite{dimas2004sym,dimas2006new}). Note that this procedure is
  described in further detail in \cite{PapamikosPryer:2019}.

  \begin{proposition}[Infinitesimal invariance]
    The infinitesimal invariance condition is equivalent to the
    following system of $9$ equations:
    \begin{eqnarray}
      \label{eq:deteq1a}
      0
      =
      \xi^1_{u}
      =
      \xi^2_{x}
      =
      \xi^2_{u}
      =
      \eta_{x}
      =
      \eta_{tu}
      =
      \eta_{uu}
      =
      \eta + u\xi^2_u
      =
      \xi^1_t + y\xi^1_x
      =
      u\eta_u - u\xi^1_x - \eta.
    \end{eqnarray}
    Solutions of the overdetermined system of linear PDEs
    \eqref{eq:deteq1a} will yield the algebra of
    the symmetry generators \eqref{eq:vf} of the $r$-CH equation.
  \end{proposition}
  
  Given \eqref{eq:deteq1a} form an overdetermined
  system of linear partial differential equations it is possible that
  they only admit the trivial solution $\xi^1=\xi^2=\eta=0$. This
  would imply that the only Lie symmetry of the $r$-CH equation is the
  identity transformation. In what follows we will see that this is
  not the case. We are able to obtain the Lie algebra for the symmetry
  generators the $r$-CH equation and thus, using the Lie series,
  derive the groups of Lie point symmetries.
  
  \begin{proposition}[Determining equations]
    The general solution of the determining equations
    \eqref{eq:deteq1a} is given by
    \begin{equation}
      \label{eq:sol-gena}
      \xi^1= c_1,
      \quad
      \xi^2= c_2 - c_3 t,
      \quad
      \eta= c_3 u,
    \end{equation}
    where $c_i$, $i=1, 2, 3$ are arbitrary real constants. 
  \end{proposition}

  \begin{proposition}[Lie algebra generators]  
    \label{pro:Lie-algebra-gen}
    It follows that the solution \eqref{eq:sol-gena} defines a three
    dimensional Lie algebra of generators where a basis is formed by
    the following vector fields
    \begin{eqnarray}
      &X_1
      =
      \frac{\partial }{\partial x},
      \quad
      X_2
      =
      \frac{\partial }{\partial t},
      \quad
      X_3
      =
      u\frac{\partial }{\partial u}
      -
      t\frac{\partial }{\partial t}.
    \end{eqnarray}
  \end{proposition}

\subsection*{Invariant solutions through symmetry reductions for $r \in 2\naturals$}

We now state solutions that occur through symmetry reductions of
(\ref{eq:fi}) to ODEs by means of the algebra generators given in
Proposition \ref{pro:Lie-algebra-gen}. We consider each generator
seperately and examine some examples of solutions from each.

\subsection*{$X_1$}
To begin notice that solutions of (\ref{eq:fi}) that are
invariant under the symmetry generated by $X_1$ are of the form $u =
f(t)$, which immediately yields $u\equiv \text{const}$ as a trivial
solution prescribed by the initial condition.

\subsection*{$X_2$}
Solutions invariant under the symmetry generated by $X_2$ are of the
form $u = f(x)$. The reduced equation is given by the following ODE:
\begin{equation}
  \label{eq:X2ode}
  \qp{1+r} f(x)^rf'(x)^4
  =
  \qp{r-1}f(x)f'(x)^r
  \qp{
    2f'(x)^2f''(x) + \qp{r-2}f(x)f''(x)^2+f(x)f'(x)f'''(x)
    }.
\end{equation}
The general solution of this ODE is not known \response{although the
  $r=2$ case is known to reduce to an elliptic integral.

In addition it is clear that the equation has a first
integral. Indeed, noticing that (\ref{eq:X2ode}) is a total
derivative one has that
\begin{equation}
  f(x)\qp{f(x)^r - f(x) f'(x)^{r-2} f''(x)} = K,
\end{equation}
for some constant $K$.
}

\subsection*{$X_3$}
Solutions invariant under the symmetry generated by $X_3$ are of the
form $u = f(x) t^{-1}$, for non-constant $f$. The reduced ODE is given by:
\begin{equation}
  \begin{split}
    \qp{\frac{f(x)}{t}}^r f'(x)^3
    \qp{1-r + \qp{r+1} f'(x)}
    &=
    \qp{r-1}f(x)\qp{\frac{f'(x)}{t}}^r \big(f''(x)\big(f'(x)\qp{1-r+2f'(x)}
        \\
        &\qquad + \qp{r-2}f(x)f''(x) + f(x)f'(x)f'''(x)\big)\big)
  \end{split}
\end{equation}
whose general solution is not known. \response{In the case $r=2$ the
  ODE can be solved in terms of Painlev\'e transcendents
  \citep{barnes2022similarity}.}

\subsection*{Travelling wave solutions}

\response{ One particular consequence of the symmetries given in the
  previous section is that the equation admits travelling wave
  solutions of the form $u(x,t) = f(x-ct)$ for constant speed $c$. In
  the case $r=2$ the reduced ODE is given as
  \begin{equation}
    2 f'(\xi) f''(\xi) + f(\xi) f'''(\xi) - 3 f(\xi) f'(\xi) + c f'(\xi) - c f'''(\xi) = 0
  \end{equation}
  and the solution is the celebrated peakon
  \begin{equation}
    \label{eq:peakon}
    u(x,t) = c \exp\qp{-\norm{x-ct}}.
  \end{equation}

  It is interesting to note that the result nontrivially extends to
  the general $r$-CH equation. In the general setting the reduced ODE
  is given by:
  \begin{equation}
    \begin{split}
      0 &= f(\xi)^{r-2} \qp{c - cr + \qp{1+r} f(\xi)} f'(\xi)
      +
      f'(\xi)^{r-3} f''(\xi) \qp{-2 f'(\xi)^2 + \qp{r-2}\qp{c-f(\xi)}f''(\xi)}
      \\
      & \qquad +
      \qp{c-f(\xi)}f'(\xi)^{r-2}f'''(\xi),
    \end{split}
  \end{equation}
  which also admits peakon solutions of the same form as
  (\ref{eq:peakon}). In particular, the travelling wave solution 
  for r-CH is the singular single $N=1$ peakon solution given as equation 
  \eqref{TW_solution} discussed in the next section.}

\begin{remark}
  It is interesting to consider the $r\to 1$ limit of Equation
  \eqref{eq:fi}. To do this, we first note that
  \begin{equation}
    \left(|u_x|^{r-2}u_x\right)_t = (r-1)|u_x|^{r-2}u_{xt},
    \quad \mbox{ and }
    \left(|u_x|^{r-2}u_x\right)_x = (r-1)|u_x|^{r-2}u_{xx}.
  \end{equation}
  Substituting into \eqref{eq:fi} gives
  \begin{align}
    0 &=    \psi_t - |u_x|^{r-2}u_{xt} + \frac{1+r}{r}|u|^r
    - |u_x|^{r-2}uu_{xx} - \frac{1}{r-1}|u_x|^r + \frac{1}{r(r-1)}|u_x|^r, \\
      &=    \psi_t - |u_x|^{r-2}u_{xt} + \frac{1+r}{r}|u|^r
    - |u_x|^{r-2}uu_{xx} - \frac{1}{r}|u_x|^r.
  \end{align}
  Now taking the limit $r\to 1$ gives
  \begin{equation}
    \psi_t - \frac{u_{xt}}{|u_x|} + 2|u| - \frac{uu_{xx}}{|u_x|} - |u_x| = 0,
  \end{equation}
  where
  \begin{equation}
    \psi(x) = \int_{-\infty}^x \sgn(u)\diff x'.
  \end{equation}
  Assuming that the sign of $u$ is globally constant in space and time,
  this is equivalent to the equation
  \begin{equation}
    \label{eq:1ch}
    u_{xt} + uu_{xx} - 2|u||u_x| - |u_x|^2 = 0.
  \end{equation}
\end{remark}

\subsection*{Invariant solutions through symmetry reductions for $r=1$}

The Lie algebra generators given in Proposition
\ref{pro:Lie-algebra-gen} still hold for $r<2$. Indeed, the case $r=1$
yields the following reductions:

\subsection*{$X_1$}
To begin notice that solutions of (\ref{eq:1ch}) that are
invariant under the symmetry generated by $X_1$ are of the form $u =
f(t)$, which immediately yields $u\equiv \text{const}$ as a trivial
solution prescribed by the initial condition.

\subsection*{$X_2$}
Solutions invariant under the symmetry generated by $X_2$ are of the
form $u = f(x)$. The reduced equation is given by the following ODE:
\begin{equation}
  f'(x)^2 + 2\norm{f(x)}\norm{f'(x)} = f(x)f''(x),
\end{equation}
which has general solution
\begin{equation}
  f(x) = \exp\qp{\tfrac 12 \qp{\exp\qp{2x + 2c_2 - \tfrac{c_1}2}}}.
\end{equation}

\subsection*{$X_3$}
Solutions invariant under the symmetry generated by $X_3$ are of the
form $u = f(x) t^{-1}$, for non-constant $f$. The reduced ODE is given by:
\begin{equation}
  f'(x) + f'(x)^2 + 2\norm{f(x)}\norm{f'(x)} = f(x)f''(x), 
\end{equation}
which has solution for $x>1$
\begin{equation}
  \begin{split}
    f(x) &= w(x) \qp{x+c_2}
    \\
    w^{-1}(x) &= \int_1^x \frac{1}{c_1 s + s\log{s} - 1} \diff s.
  \end{split}
\end{equation}

\section{Singular solutions}\label{sec: singsolns}
\label{sec:singular}
In this section, we construct evolutionary singular solutions that
will turn out to be weak solutions of the integrated form of r-CH in
\eqref{eq:fi}. 

\begin{definition}[singular solutions]
  Let $-\infty=Q_0 < Q_1 < Q_2 < \ldots < Q_N < Q_{N+1}=\infty$.
  For given $P=(P_1, P_2, \ldots, P_N)$,
  we define $\tilde{u}(x,t;P,Q)\in \sob{1}{r}$ such that
  \begin{equation}
    \label{eq:tilde u}
      \int_{-\infty}^{\infty} |\tilde{u}|^{r-2}\tilde{u}v + \frac{1}{r-1}|\tilde{u}_x|^{r-2}\tilde{u}_xv_x \diff x = \sum_{i=1}^NP_iv(Q_i), \quad \forall v\in \sob{1}{r}.
  \end{equation}
\end{definition}
\begin{proposition}
  The singular solution $\tilde{u}(x,t;P,Q)$ has the following properties.
\begin{enumerate}
  \item It satisfies
    \begin{equation}
      \label{eq:helmholtz}
      |\tilde{u}|^{r-2}\tilde{u} - \frac{1}{r-1}(|\tilde{u}_x|^{r-2}\tilde{u}_x)_x = 0\,,
    \end{equation}
    in each interval $\Omega_i = (Q_i,Q_{i+1})$, $i=0,\ldots,N$,
  \item It is continuous at $x=Q_i$, $i=1,\ldots,N$ (but without
    continuous derivatives there in general),
  \item The discontinuities in the derivatives at the same points
    are specified by
      \begin{equation}
    \label{eq:Pdef}
    -\frac{1}{r-1}\jump{|\tilde{u}_x|^{r-2}\tilde{u}_x}_{Q_i} = P_i, \quad i=1,\ldots,N,
      \end{equation}
        where
  \begin{equation}
  \jump{f(x)}_y = \lim_{x\to y^+}f(x) - \lim_{x\to y^-}f(x).
  \end{equation}
\item There exist constants $C_1,C_2,\ldots, C_N$ such that
  \begin{equation}
    \label{eq:ci}
    \tilde{u}^{r} = \tilde{u}_x^r + C_i\,,
  \end{equation}
  where the constant $C_i$ takes a different value in each
  interval $\Omega_i$.
\end{enumerate}
\end{proposition}
\begin{proof}
  First we note that \eqref{eq:tilde u} is well-posed, by energy
  methods, since it minimises the \response{convex} functional
  \begin{equation}
    \frac{1}{r}\int_{-\infty}^{\infty}|\tilde{u}|^r + \frac{1}{r-1}
    |\tilde{u}_x|^r \diff x - \sum_{i=1}^N P_i\tilde{u}(Q_i).
  \end{equation}
  By Sobolev embedding we can find a continuous
  representative of $\tilde{u}$ in the $\sob{1}{r}$ equivalence class,
  hence property 2 holds. 
  
  To obtain properties 1 and 3, we divide the
  integral in \eqref{eq:tilde u} into intervals and integrate by
  parts separately in each interval,
  \begin{align}
    0 & = \sum_{i=0}^N\int_{Q_i}^{Q_{i+1}}
    |\tilde{u}|^{r-2}\tilde{u}v + |\tilde{u}_x|^{r-2}\tilde{u}_xv_x \diff x
    - \sum_{i=1}^N P_iv(Q_i), \\
    & = \sum_{i=0}^N\int_{Q_i}^{Q_{i+1}}
    \left(|\tilde{u}|^{r-2}\tilde{u} - \left(|\tilde{u}_x|^{r-2}\tilde{u}_x\right)_x\right)v \diff x
    - \sum_{i=1}^N \left(\jump{|u_x|^{r-2}u_x}_{Q_i}+P_i\right)v(Q_i),
  \end{align}
  and the results follow from standard arguments (using e.g. a
  shooting argument to ensure existence of second derivatives of
  $\tilde{u}$).  
  
    To obtain property 4, make the substitution $w=u_x$,
     to obtain
  \begin{equation}
    |u|^{r-2}u = \frac{1}{r-1} \left(|w|^{r-2}w\right)_x
    = \frac{1}{r-1}\frac{d}{du}\left(|w|^{r-2}w\right)\frac{du}{dx}
    = |w|^{r-2}\frac{dw}{du},
\end{equation}
which we may integrate to get \eqref{eq:ci}.
\end{proof}
The $C_i$ can be determined by integrating \eqref{eq:ci}, after careful
consideration of the sign of $u_x$. This is discussed further in
Appendix \ref{app:Ci}.

We now define a Hamiltonian system
for $P$ and $Q$.
\begin{definition}[Hamiltonian system]
  Given $Q,P$, we define the Hamiltonian
  $\hat{H}(P,Q)=H(\tilde{u}(\cdot;P,Q))$, i.e., we compute
  $\tilde{u}(x,t;Q,P)$ and substitute into \eqref{rCH-Ham}.
\end{definition}
Hamilton's equations with this Hamiltonian have a very simple form.
\begin{proposition}
  Hamilton's canonical equations with this Hamiltonian may be equivalently written as 
  \begin{align}  
 \label{eq:pfdQ}
    \dot{Q}_i & = \tilde{u}_i(Q_i, t; P,Q), \\ \label{eq:pfdP}
    \dot{P}_i & = {\frac{1}{r}}\jump{|u_x|^r}_{Q_i}, \quad i=1,2,\ldots,N.
  \end{align}
\end{proposition}
In fact, this dynamics is fully realisable as a finite dimensional system,
since we have
\begin{equation}
  \jump{u_x^r}_{Q_i} = \left(\widehat{u}_i^r - C_{i+1} -
  \left(\widehat{u}_i^r - C_{i}\right)\right) = (C_{i+1} -
  C_{i}),
\end{equation}
where $C_i$, $i=1,\ldots,N$ can be obtained numerically using the formulae
in Appendix \ref{app:Ci}.
\begin{proof}
  First, note that we may write
  \begin{equation}
    \hat{H}(P,Q) = \sum_{i=1}^N P_i\tilde{u}(Q_i,t;Q,P)
    - \int_{-\infty}^{\infty} |\tilde{u}|^r + |\tilde{u}_x|^r\diff x.
  \end{equation}
  To see this, note that
  \begin{equation}
    \sum_{i=1}^N P_i\tilde{u}(Q_i,t;Q,P)
    = \int_{-\infty}^{\infty} |\tilde{u}|^{r-2}\tilde{u}^2 + |\tilde{u}_x|^{r-2}\tilde{u}_x^2 \diff x,
  \end{equation}
  after taking $v=\tilde{u}(\cdot, t; P,Q)$ in \eqref{eq:tilde u}.
  Then,
  \begin{align}
    \dot{Q}_i = \pp{\hat{H}}{P_i} & =
    \tilde{u}(Q_i, t; P, Q)
    + \sum_{j=1}^N P_j\pp{\tilde{u}}{P_i}(Q_j, t; P, Q)
    - \sum_{j=0}^N \int_{Q_j}^{Q_{j+1}}
    |\tilde{u}|^{r-2}\tilde{u}\pp{\tilde{u}}{P_i}
    + |\tilde{u}_x|^{r-2}\tilde{u}_x\pp{}{x}\pp{\tilde{u}}{P_i}\diff x, \\
     & =
    \tilde{u}(Q_i, t; P, Q)
    + \sum_{j=1}^N
    \underbrace{\left(P_j + \jump{|\tilde{u}|^{r-2}\tilde{u}}_{Q_j}\right)}_{=0}\pp{\tilde{u}}{P_i}(Q_j, t; P, Q) \nonumber \\
    & \qquad
    - \sum_{j=0}^N \int_{Q_j}^{Q_{j+1}}
    \underbrace{\left(|\tilde{u}|^{r-2} \tilde{u}
    - \left(|\tilde{u}_x|^{r-2}\tilde{u}_x\right)_{x}\right)}_{=0}
    \pp{\tilde{u}}{P_i}\diff x, \\
     & =
    \tilde{u}(Q_i, t; P, Q),    
  \end{align}
  as required. Further,
  \begin{align} \nonumber
    \dot{P}_i = -\pp{\hat{H}}{Q_i} = &
    - \sum_{j=1}^N P_j\pp{\tilde{u}}{Q_i}(Q_j, t; P, Q)
    + \sum_{j=0}^N \int_{Q_j}^{Q_{j+1}}
    |\tilde{u}|^{r-2}\tilde{u}\pp{\tilde{u}}{Q_i} 
    + |\tilde{u}_x|^{r-2}\tilde{u}_x\pp{}{x}\pp{\tilde{u}}{Q_i}\diff x \\
    & \qquad + \frac{1}{r}\left(|\tilde{u}|^r+\frac{1}{r-1}|\tilde{u}_x|^r
    \right)|_{Q_i^+}
    -\frac{1}{r}\left(|\tilde{u}|^r+\frac{1}{r-1}|\tilde{u}_x|^r
    \right)|_{Q_i^-} \\
    =
     & 
    -\sum_{j=1}^N
    \underbrace{\left(P_j + \jump{|\tilde{u}|^{r-2}\tilde{u}}_{Q_j}\right)}_{=0}\pp{\tilde{u}}{P_i}(Q_j, t; P, Q)
    + \sum_{j=0}^N \int_{Q_j}^{Q_{j+1}}
    \underbrace{\left(|\tilde{u}|^{r-2} \tilde{u}
    - \left(|\tilde{u}_x|^{r-2}\tilde{u}_x\right)_{x}\right)}_{=0}
    \pp{\tilde{u}}{Q_i}\diff x \\
    & \qquad + \frac{1}{r}\jump{|u_x|^{r}}_{Q_i},
  \end{align}
  as required.
\end{proof}
\begin{remark}
  A consequence of these equations is that the sign of $P_i$ is
  preserved, for $i=1,\ldots,N$. To see this, note that if $P_i=0$,
  then $\jump{u_x^{r-1}}|_{Q_i}=0$, which means that
  $r\dot{P}_i=\jump{u_x^r}|_{Q_i}=0$. That is, if $P_i=0$ then it remains
  zero for all times by continuity.
\end{remark}
The following theorem explains the connection between the singular
solutions and the r-Camassa-Holm equation.
\begin{theorem}
  Let $u=\tilde{u}(x,t;P,Q)$ satisfy the singular solution dynamics above. Then
  $u$ is a weak solution of Equation \eqref{eq:fi}.
\end{theorem}
\begin{proof}
  Let $\phi(x)$ be a $C^\infty$ test function that vanishes at infinity,
  and let $\Phi$ be such that $\Phi_x=\phi$.
  Then
  \begin{align}
    \begin{split}
    \dd{}{t}\int_{-\infty}^{\infty}
    \phi\left(\psi - \frac{1}{r-1}|u_x|^{r-2}u_x\right)\diff x & =
    \dd{}{t}\int_{-\infty}^{\infty}
    \Phi_x\left(\psi - \frac{1}{r-1}|u_x|^{r-2}u_x\right)\diff x, \\
[\mbox{integration by parts}]    & =
    -\dd{}{t}\int_{-\infty}^{\infty}
    \Phi |u|^{r-1}u + \frac{1}{r-1} \Phi_x |u_x|^{r-2}u_x\diff x, \\
    & = -\dd{}{t}\sum_{i=1}^N P_i\Phi(Q_i), \\
    & = -\sum_{i=1}^N\left(\dot{P}_i\Phi(Q_i) + P_i\phi(Q_i)u(Q_i)\right), \\
    & = -\sum_{i=1}^N\left(\frac{1}{r}\jump{|u_x|^r}_{Q_i}\Phi(Q_i)
    - \frac{1}{r-1}\jump{|u_x|^{r-2}u_x}_{Q_i}\phi(Q_i)u(Q_i)\right), \\
    & = -\sum_{i=1}^N\left(\frac{1}{r}\jump{|u_x|^r\Phi}_{Q_i}
    - \frac{1}{r-1}\jump{|u_x|^{r-2}u_x\phi u}_{Q_i}\right), \\
    & = -\sum_{i=0}^N\int_{Q_i}^{Q_{i+1}}\frac{-1}{r}(|u_x|^r\Phi)_x
    + \frac{1}{r-1}(|u_x|^{r-2}u_x\phi u)_x\diff x, \\
    & = -\sum_{i=0}^N\int_{Q_i}^{Q_{i+1}}\frac{-1}{r}\left((|u_x|^r)_x\Phi +
    u_x^r\phi\right) \\
    & \qquad \qquad
    + \frac{1}{r-1}\left((|u_x|^{r-2}u_x)_x \phi u + |u_x|^r\phi +
    |u_x|^{r-2}u_x\phi_xu\right) \diff x.
    \end{split}
    \end{align}
  Restricted to the interval $\Omega_i$, we have
  $|u|^r = |u_x|^r + C_i$, and so $(|u_x|^r)_x = (|u|^r)_x$.
  We also have $|u|^{r-2}u=\frac{1}{r-1}(|u_x|^{r-2}u_x)_x$.
  Hence,
    \begin{align}
    \begin{split}
    \dd{}{t}\int_{-\infty}^{\infty}
    \phi\left(\psi - \frac{1}{r-1}|u_x|^{r-2}u_x\right)\diff x & =
    -\sum_{i=0}^N\int_{Q_i}^{Q_{i+1}}\frac{-1}{r}(|u|^r)_x\Phi - \frac{1}{r}
    |u_x|^r\phi \\
    & \qquad + \frac{1}{r-1}
    \left((|u_x|^{r-2}u_x)_x \phi u + |u_x|^r\phi + |u_x|^{r-2}u_x\phi_xu\right) \diff x, \\
    & \label{eq:parts} =
    -\sum_{i=0}^N\int_{Q_i}^{Q_{i+1}} \frac{1}{r}|u|^r\phi - \frac{1}{r}
    |u_x|^r\phi \\
    & \qquad 
    + \frac{1}{r-1}
    \left(\underbrace{(|u_x|^{r-2}u_x)_x}_{=(r-1)|u|^{r-2}u} \phi u
    + |u_x|^r\phi + |u_x|^{r-2}u_x\phi_xu\right) \diff x, \\
    & = -\int_{-\infty}^{\infty}\frac{r+1}{r}|u|^r\phi + \frac{1}{r(r-1)}
    |u_x|^r\phi + \frac{1}{r-1}|u_x|^{r-2}u_x\phi_xu \diff x,
    \end{split}
    \end{align}
    where a global integration by parts in the first term in line
    \eqref{eq:parts} was possible since $\Phi$ and $u$ are both
    globally continuous.  Finally we get
    \begin{equation}
      \int_{-\infty}^{\infty}
      \phi\left(\psi - \frac{1}{r-1}|u_x|^{r-2}u_x\right)_t +
      \frac{r+1}{r}|u|^r\phi + \frac{1}{r(r-1)}
    |u_x|^r\phi + \frac{1}{r-1}|u_x|^{r-2}u_x\phi_xu \diff x = 0,
    \end{equation}
    for all test functions $\phi$, and we obtain the result
    by passing to the limit $\phi \to \sob{1}{r}$.
\end{proof}

\section{Examples}
\label{sec:solns}

This section discusses the results of numerical simulations of 2-point
collisions both overtaking and antisymmetric (head-on) and 3-point
overtaking collisions.

\subsection{1 point solution}
In the case $N=1$, the vanishing boundary conditions at $\pm\infty$
mean that $C_0=C_1=0$, and we have $u^r - (r-1)u_x^r = 0$ in both intervals,
which has the solution $u=\widehat{u}_1\exp(-|x-Q_1|)$. Then,
\begin{equation}
  P_1 = -\jump{u_x^{r-1}}|_{x=Q_1}=2\widehat{u}_1^{r-1},
\end{equation}
and
\begin{equation}
  \dot{P}_1 = \frac{1}{r}\jump{u_x^r}_{Q_1} = 0,
\end{equation}
leading to the \emph{travelling wave solution}
\begin{equation}
  u(x,t) = \widehat{u}_1\exp(-|x-Q_1|), \quad Q_1(t) = Q_1(0)+\widehat{u}_1t.
\label{TW_solution}
\end{equation}
The peaked exponential shape of the 1 point solution will re-emerge whenever the multi-point solutions are well separated
compared to the width of the exponential. 

\begin{figure}
    \includegraphics[width=8cm]{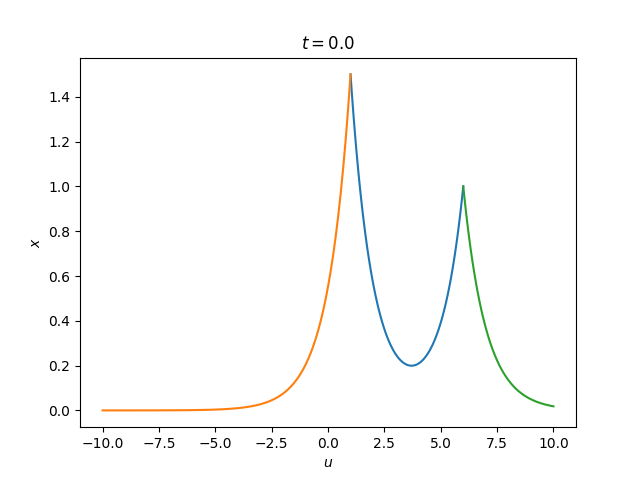}
    \includegraphics[width=8cm]{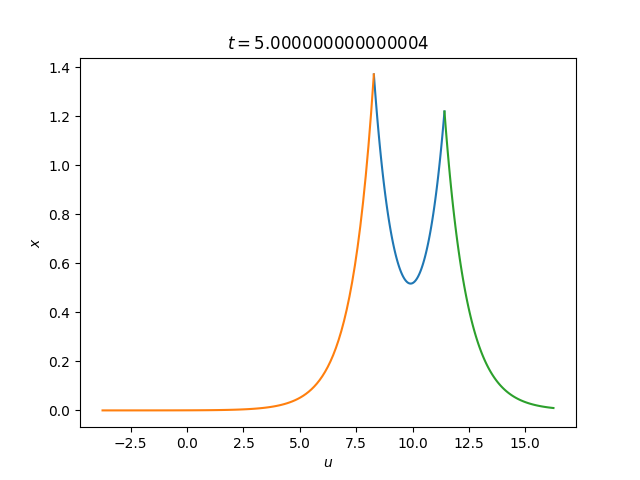}\\
    \includegraphics[width=8cm]{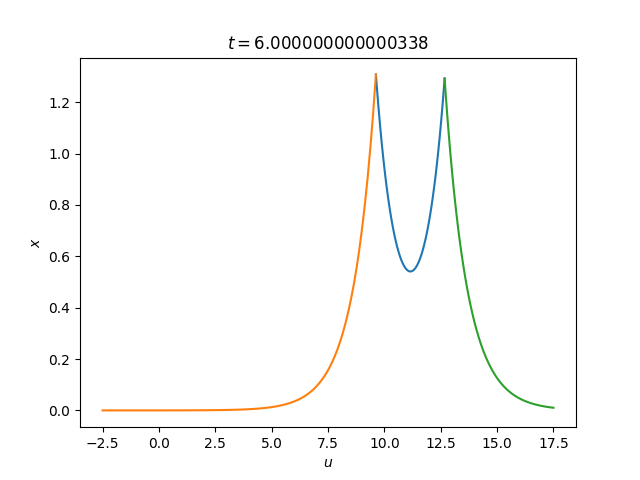}
    \includegraphics[width=8cm]{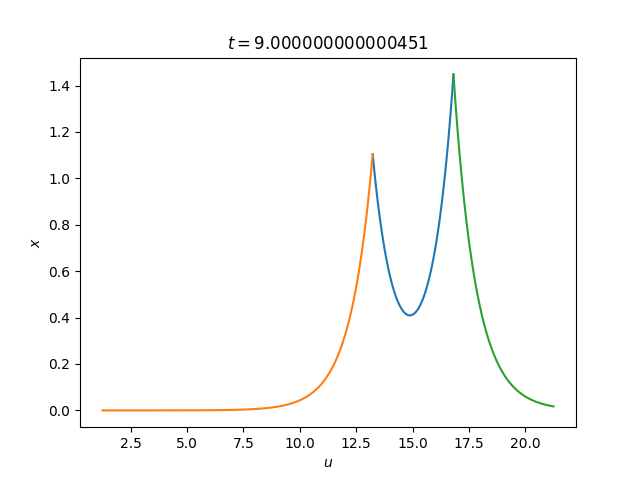}
\caption{Plots of $u(x)$ from the overtaking collision for $r=2$ at various times. Top-left: $t=0$. Top-right: $t=3$. Bottom-left: $t=5$. Bottom-right: $t=9$. The peakon on the left has a higher velocity and so it catches up with the peakon on the right, transferring momentum through nonlocal interactions until the peakon on the right is moving faster and they separate again.}.
    \label{fig:overtaking r2}
\end{figure}

\begin{figure}
    \centering
    \includegraphics[width=8cm]{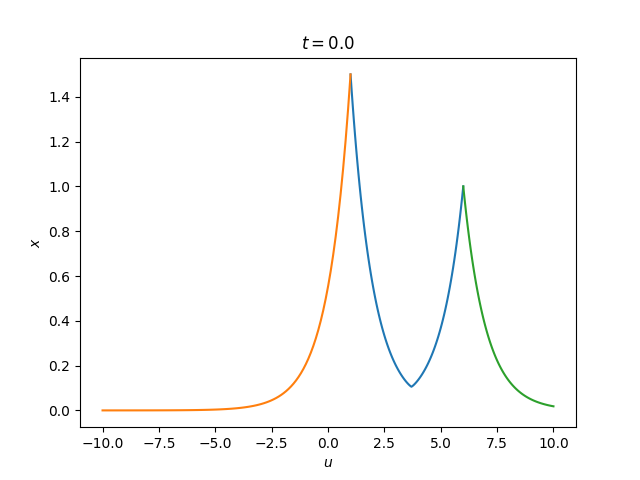}
    \includegraphics[width=8cm]{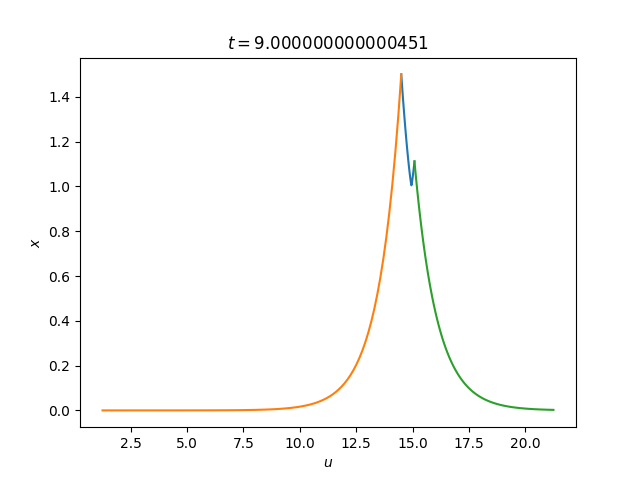}\\
    \includegraphics[width=8cm]{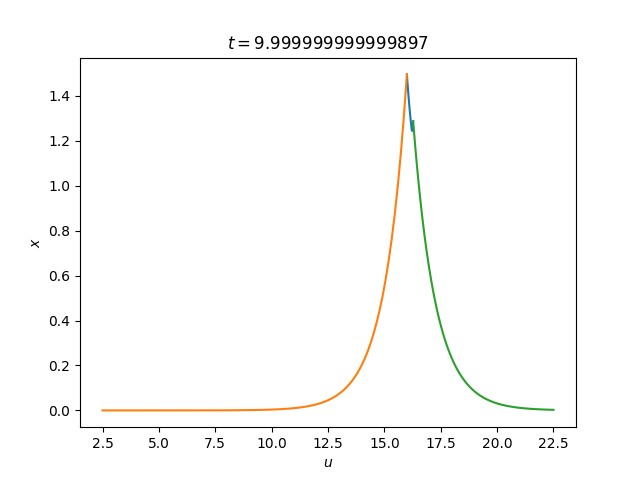}
    \includegraphics[width=8cm]{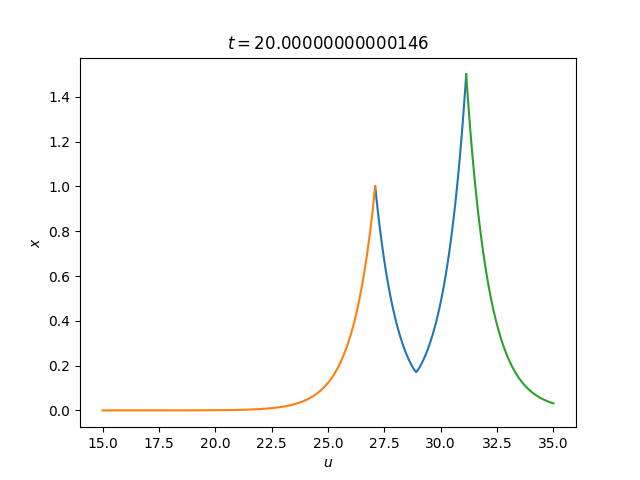}
\caption{Plots of $u(x)$ from the overtaking collision for $r=6$ at various times. Top-left: $t=0$. Top-right: $t=9$. Bottom-left: $t=10$. Bottom-right: $t=20$. The main observable differences are that there is a much sharper minimum between the two peaks for $r=6$ than for $r=2$ in the previous figure, and the two peaks get much closer to each other during the collision for $r=6$ than for $r=2$.}.
    \label{fig:overtaking r6}
\end{figure}

\begin{figure}
    \centering
    \includegraphics[width=8cm]{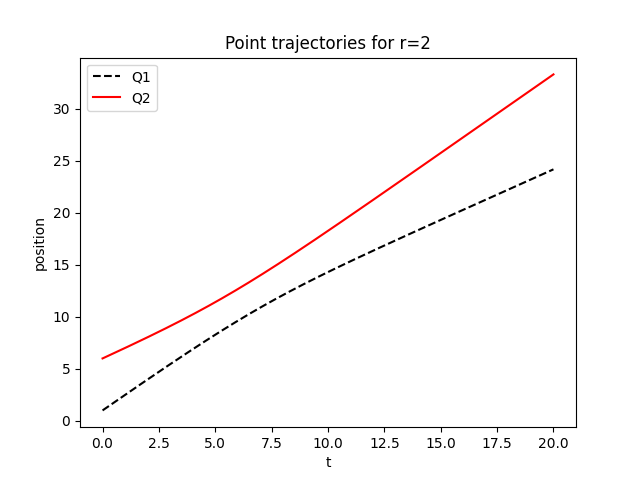}
    \includegraphics[width=8cm]{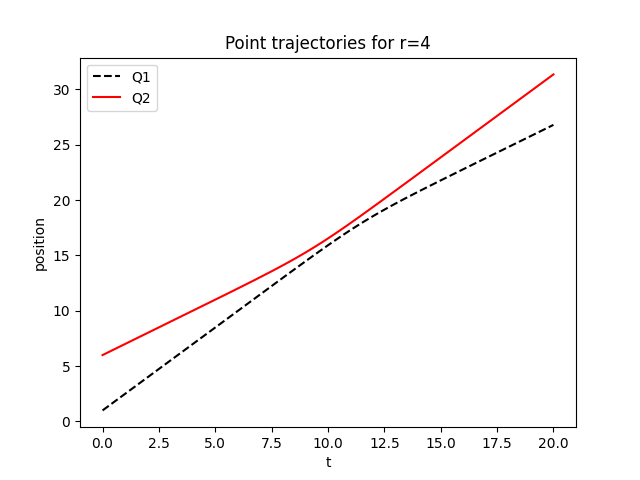}
\includegraphics[width=8cm]{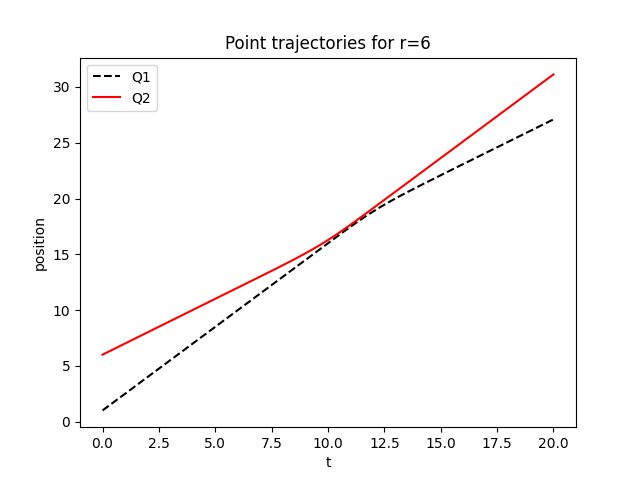}
\caption{Plots of $Q_1$ and $Q_2$ versus time for various $r$: Top-left: $r=2$. Top-right: $r=4$. Bottom: $r=6$. For larger $r$, we see that the peaks are behaving much more like billiard balls, with the elastic interaction being much more local and therefore producing much less phase shift. This is because in the $\sob{1}{r}$ norm, a smaller reduction in the larger peak is required to balance the same increase in the smaller peak.}
    \label{fig:rear}
\end{figure}

\begin{figure}
    \centering
    \includegraphics[width=8cm]{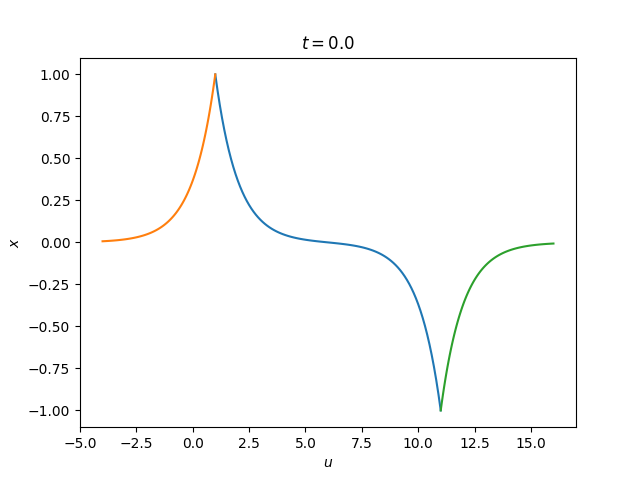}
    \includegraphics[width=8cm]{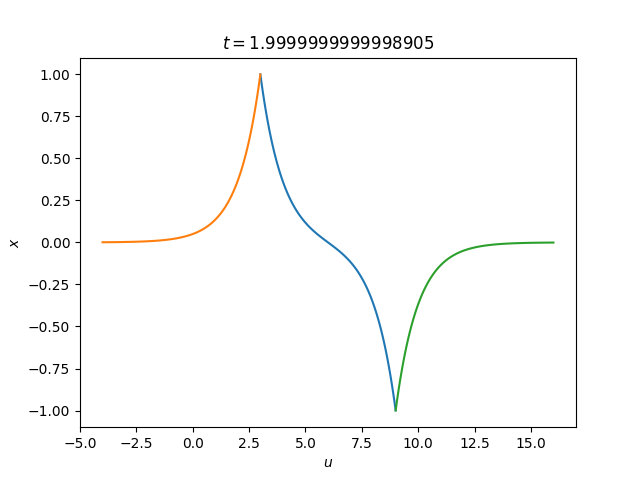}\\
    \includegraphics[width=8cm]{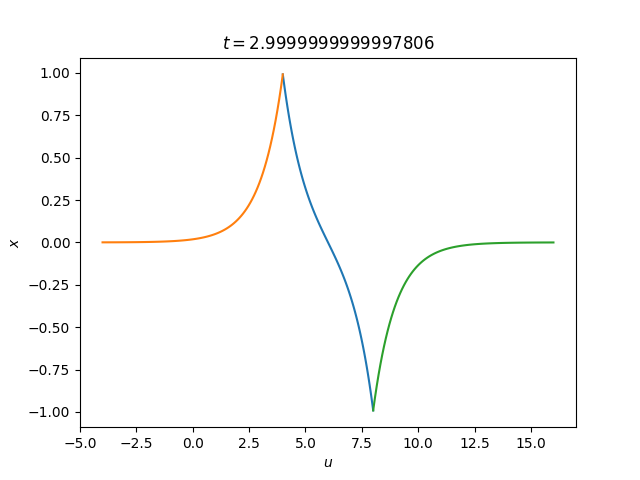}
    \includegraphics[width=8cm]{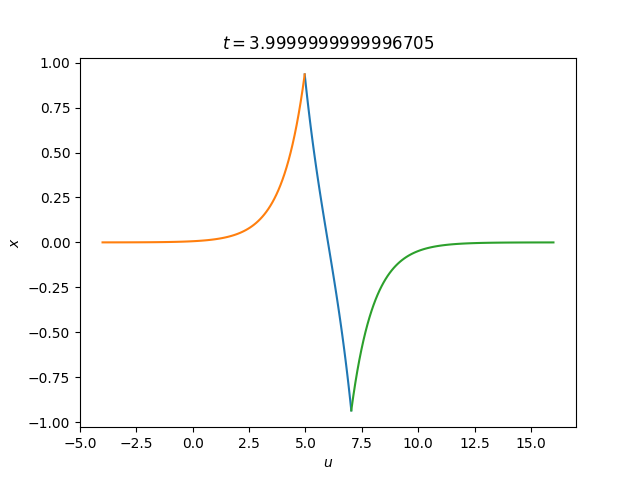}
\caption{Plots of $u(x)$ from the antisymmetric collision for $r=2$ at various times. Top-left: $t=0$. Top-right: $t=2$. Bottom-left: $t=3$. Bottom-right: $t=4$. As the peaks get closer together, the profile between the peaks approaches a straight line, with the derivative becoming more and more negative. This causes $P_1$ and $P_2$ to get larger, preventing the two peaks from cancelling each other out and leading to a collision in finite time.}.
    \label{fig:head r2}
\end{figure}

\begin{figure}
    \centering
    \includegraphics[width=12cm]{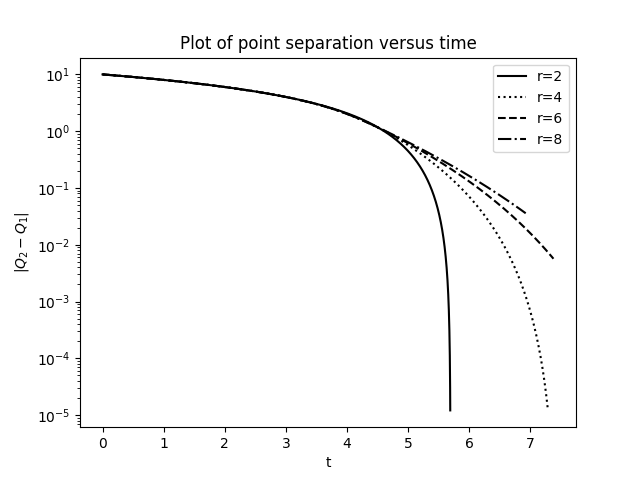}
\caption{Antisymmetric collision paths for various values of $r$ ranging from  from $r=2$ to $r=8$. We see that the collision takes the shortest time for $r=2$, and the collision time is increasing as $r$ increases from 2 to 8. This behaviour is consistent with the hypothesis that the collision time tends to infinity as $r\to \infty$.}
    \label{fig:head r}
\end{figure}


\subsection{2 point solutions}
Now we \response{solve the} $N=2$ case. It is necessary to use
numerical methods to solve the equation for the singular solution
profile in $[Q_1,Q_2]$, either by solving the quadratures in Appendix
\ref{app:Ci}, or by using a direct numerical discretisation of
\eqref{eq:helmholtz}.  In this section we compute peakon solutions by
solving \eqref{eq:helmholtz} numerically in each interval
$[Q_i,Q_{i+1}]$, using a finite element discretisation, implemented
using Firedrake \citep{rathgeber2016firedrake}.

For two point solutions, the range $[Q_1,Q_2]$ is transformed to the
unit interval $[0,1]$, to facilitate symbolic differentiation to
compute the Hamiltonian vector field. Then we can write
\begin{equation}
L[u,Q] = \hat{L}[\hat{u},Q] = \frac{1}{r}\Delta Q\int_0^1 \hat{u}^r + \frac{1}{(r-1)\Delta Q^r}\hat{u}_s^r
\diff s + \frac{1}{r(r-1)}(\hat{u}(0))^r
+ \frac{1}{r(r-1)}(\hat{u}(1))^r,
\end{equation}
where $\Delta Q=Q_2-Q_1$, and $u(Q_1 + s\Delta Q)=\hat{u}(s)$,
having made use of the identities,
\begin{equation}
u(x) = \left\{
\begin{array}{cc}
\hat{u}(1)e^{Q_2-x}     & x>Q_2 \\
\hat{u}(0)e^{x-Q_1}     & x<Q_1 \\
\end{array}\right. ,
\end{equation}
which leads to
\begin{align}
    \frac{1}{r}\int_{Q_N}^{\infty}
    u^r + \frac{u_x^r}{r-1}\diff x
    & = \underbrace{\frac{1+\frac{1}{r-1}}{r}}_{=\frac{1}{r-1}}
    \hat{u}(1)^r\int_{Q_N}^{\infty}
       e^{r(Q_N-x)}\diff x, \\
       & = -\frac{1}{r(r-1)}(\hat{u}(1))^r
       [e^{r(Q_N-x)}]_{Q_N}^{\infty}
        = \frac{1}{r(r-1)}(\hat{u}(1))^r,
\end{align}
and
\begin{align}
    \frac{1}{r}\int_{-\infty}^{Q_1}
    u^r + \frac{u_x^r}{r-1}\diff x
    & = \underbrace{\frac{1+\frac{1}{r-1}}{r}}_{=\frac{1}{r-1}}
    (\hat{u}(0))^r\int_{-\infty}^{Q_1}
       e^{r(x-Q_1)}\diff x, \\
       & = \frac{1}{r(r-1)}(\hat{u}(0))^r
       [e^{r(x-Q_1)}]_{Q_1}^{\infty}
        = \frac{1}{r(r-1)}(\hat{u}(0))^r.
\end{align}
The Hamiltonian is then $H(P,Q)=S(\tilde{u},P,Q)$, where
\begin{equation}
  S[P,Q,\tilde{u}] = P_1\tilde{u}(0) + P_2\tilde{u}(1) + \hat{L}[\tilde{u},Q],
\end{equation}
and where $\tilde{u}$ is constrained to solve
\begin{equation}
  \label{eq:u constraint}
  \left\langle \frac{\delta S}{\delta \tilde{u}}, \delta \tilde{u}
  \right\rangle=0,
  \quad \forall \delta\tilde{u} \in V.
\end{equation}
For the unapproximated problem $V=\sob{1}{r}([0,1])$, which we
approximate with a conforming finite element space $V_h\subset V$. We
have
\begin{equation}
  \pp{H}{Q_i} = \pp{S}{Q_i} + \underbrace{\left\langle \pp{S}{\tilde{u}}, \pp{\tilde{u}}{Q_i}
  \right\rangle}_{=0}, \quad i=1,2,
\end{equation}
provided that \eqref{eq:u constraint} is satisfied. Similarly we have
$\pp{H}{P_i}=\pp{S}{P_i}$, $i=1,2$. Hence, the Hamiltonian equations
are
\begin{equation}
  \left\langle \pp{S}{\tilde{u}}, v \right\rangle = 0, \forall v\in V_h, \,
    \dot{P}_i = -\pp{S}{Q_i}, \, \dot{Q}_i = \pp{S}{P_i},
\end{equation}
all of which can be obtained by symbolic differentation using Firedrake.

First we consider the ``overtaking" collision, where $Q_1=1$, $Q_2=6$, $u(Q_1)=1.5$, and $u(Q_2)=1$ initially. Some plots are showing in Figure \ref{fig:overtaking r2} and Figure \ref{fig:overtaking r6}, for $r=2$ and $r=6$ respectively. We observe that for larger $r$, the peaks behave more like billiards, with the collision being 
more local (i.e., the peaks get closer before transferring momentum from the peak on the left to the peak on the right). This is made clearer in Figure \ref{fig:rear}, which shows the peak trajectories for $r=2,4,6$. The same momentum transfer occurs in each case, but the peaks get closer before transferring for larger $r$.
This is because for larger $r$, a smaller reduction in the larger peak is required to balance the same increase in the smaller peak in order to preserve the $\sob{1}{r}$ norm.

Next we consider the ``antisymmetric" collision, where $Q_1=1$,
$Q_2=11$, $u(Q_1)=1$, and $u(Q_2)=-1$ initially. The classic collision
is seen in Figure \ref{fig:head r2}. If the peaks were simply
superposed, they would cancel out, but the nonlinear interaction
drives up $P_1$ (and $P_2$ towards $-\infty$), keeping
$u(Q_1)>0>u(Q_2)$, leading to a collision in finite time. Now we
examine how this finite time collision behaves for different $r$ in
Figure \ref{fig:head r}. We see a trend in increasing collision time
as $r$ increases from 4 to 10. This is consistent with the hypothesis
that the collision time should tend to infinity as $r\to
\infty$. Similar results for the r-Hunter-Saxton equation were
observed in \cite{cotter2020r}.

\subsection{3 point solutions}

We finally present a three point solution, with $r=4$, obtained using
the 3 point extension of the same finite element method used to
produce the 2 point solutions. The initial condition is chosen so that
all three points are moving in the same direction, but the 
point farthest to the left is moving at faster speed than the two points on the
right. Plots of the solution at various times are given in Figures
\ref{fig:3pt 1} and \ref{fig:3pt 2}.  It appears that the three points
have clumped together (as we see at $t=25$ in Figure \ref{fig:3pt 1}).
However, in Figure \ref{fig:3pt 2}, we see that the points separate again at
later times. A plot of the trajectories is shown in Figure
\ref{fig:3pt traj}. This plot shows that the ordering of the points is
preserved, as we expect from the fact that the points are transported
by a globally defined velocity $u$ (i.e. they are acted on by the
time-dependent diffeomorphism generated by $u$). One may notice  
that when the points are almost superposed, the peak is
only slightly higher than the left peak in the initial condition. This
is due to the 4th power in the energy as opposed to the 2nd power for CH.

\begin{figure}
    \centering
    \includegraphics[width=8.5cm]{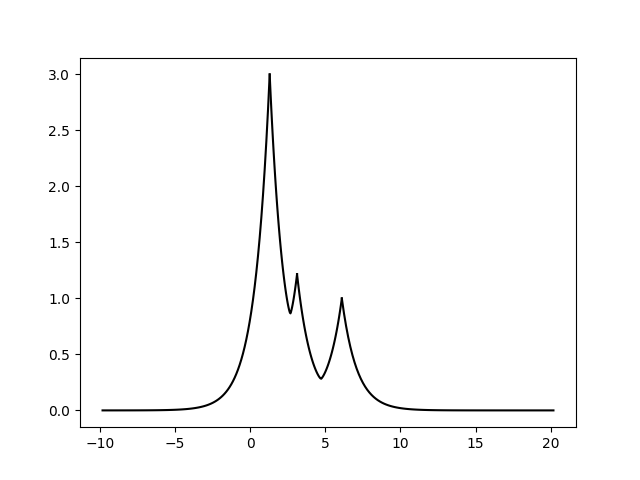}
    \includegraphics[width=8.5cm]{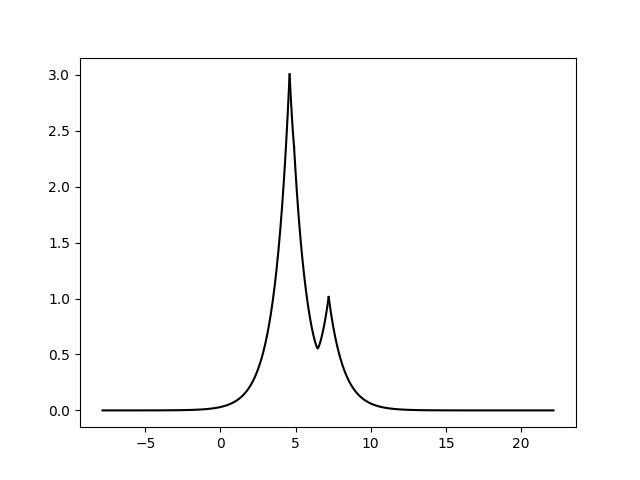}\\
    \includegraphics[width=8.5cm]{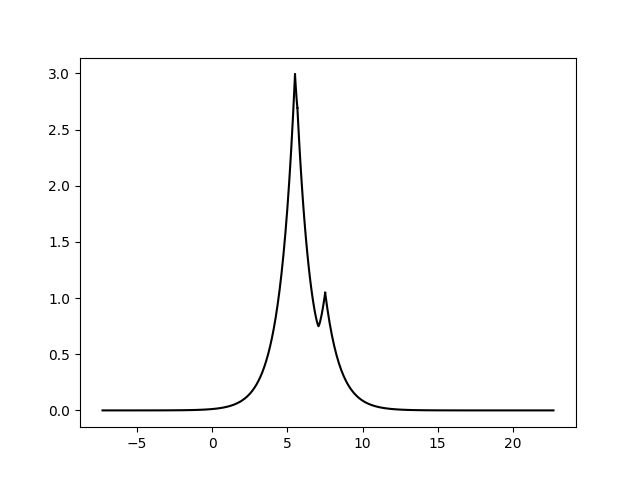}
    \includegraphics[width=8.5cm]{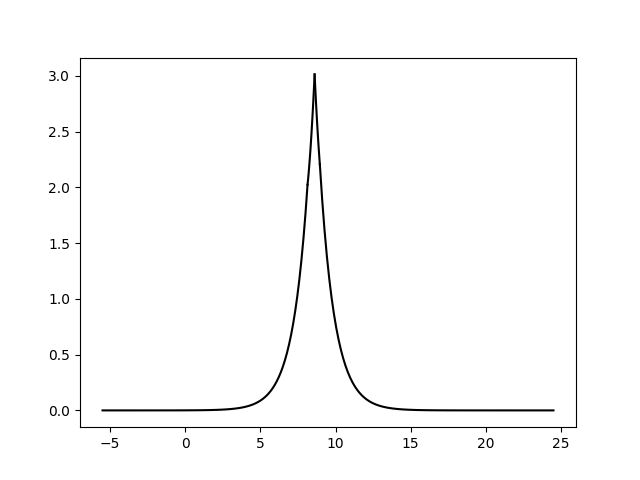}\\
    \caption{\label{fig:3pt 1}Plots of the three point solution with initial
      condition $Q_1=1$, $Q_2=3$, $Q_3=6$, $u_1=3$, $u_2=1.2$, $u_3=1$. From
      left to right, then top to bottom, the solution is plotted at times
      $t=1,12,15$ and 25 respectively. Plots at later times are shown in Figure
    \ref{fig:3pt 2}.}
\end{figure}

\begin{figure}
    \centering
    \includegraphics[width=8.5cm]{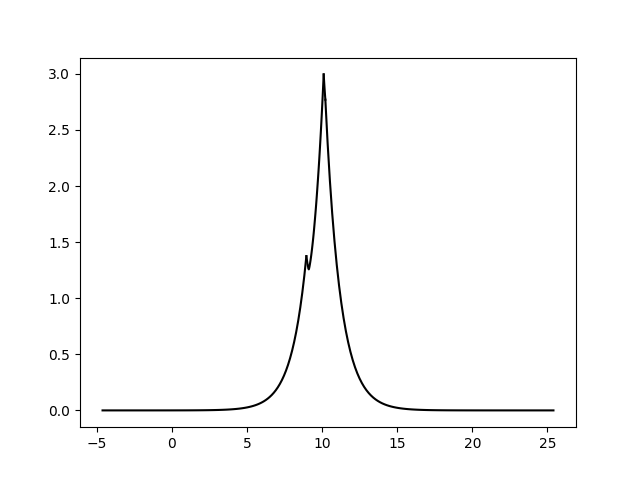}
    \includegraphics[width=8.5cm]{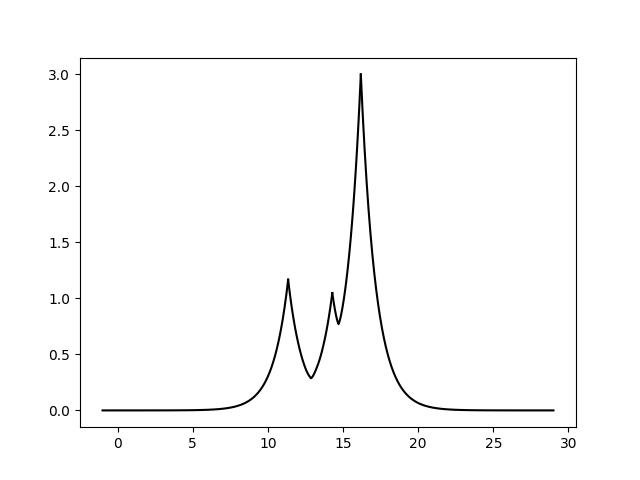}\\
    \includegraphics[width=8.5cm]{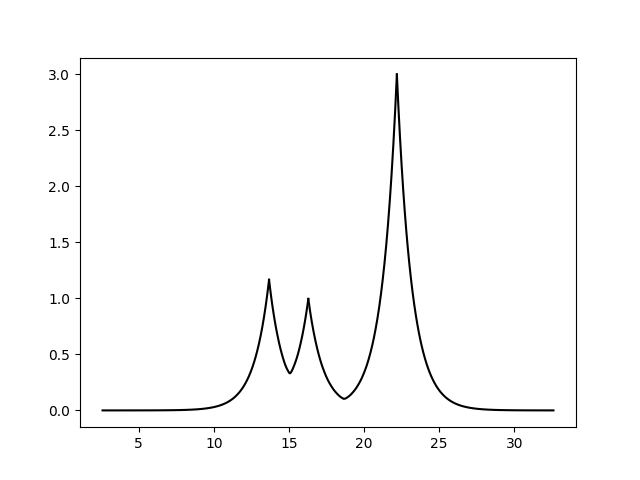}
    \includegraphics[width=8.5cm]{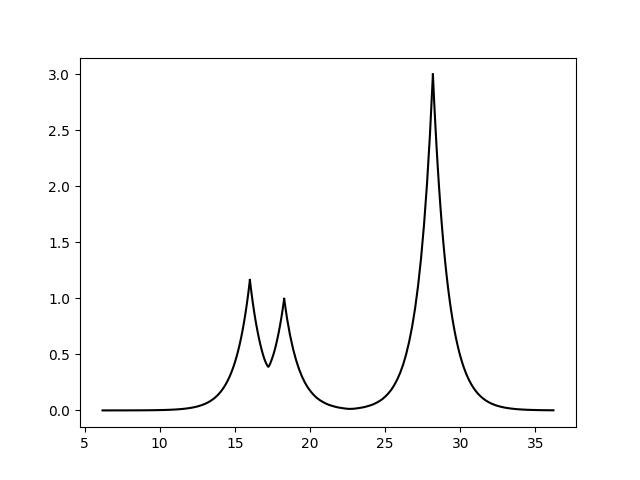}
    \caption{\label{fig:3pt 2}Plots of the three point solution with initial
      condition $Q_1=1$, $Q_2=3$, $Q_3=6$, $u_1=3$, $u_2=1.2$, $u_3=1$. From
      left to right, then top to bottom, the solution is plotted at times
      $t=30,50,70$ and 90 respectively. Earlier times are plotted in
      Figure \ref{fig:3pt 1}. Eventually the gap will widen between the two points on the left.}
\end{figure}

\begin{figure}
    \centering
    \includegraphics[width=12cm]{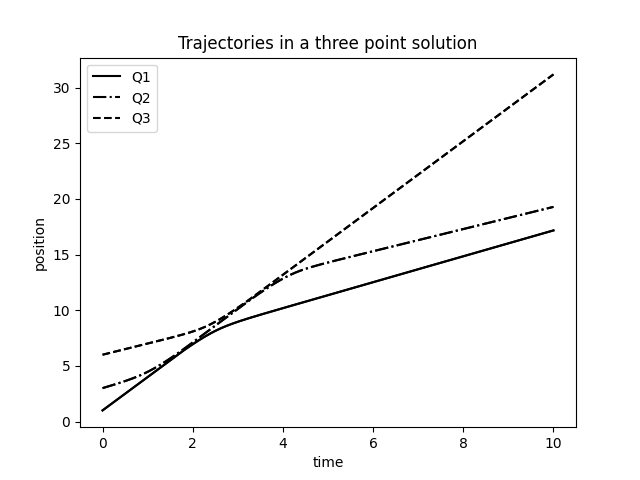}
    \caption{\label{fig:3pt traj}Point trajectories of the three point
      solution with initial condition $Q_1=1$, $Q_2=3$, $Q_3=6$,
      $u_1=3$, $u_2=1.2$, $u_3=1$. We observe that the points pass
      very close to each other but do not cross, so their order does not change.}
\end{figure}

\section{Summary and outlook}
\label{sec:summary}

Our main goal in this paper has been to demonstrate the existence of
singular weak solutions of the r-CH equation and numerically simulate
the coherence of their nonlinear interactions in overtaking collisions
and head-on collisions. Of course, many questions remain open about
the r-CH solutions. For example, we have not studied their stability,
or even the stability of their $N=1$ travelling wave. However, the
numerical simulations indicate that these solutions are likely to be
quite stable. The outcome of their head-on collisions raises some
questions, though, since wave breaking (that is, the formation of a
vertical derivative of velocity) was not seen in the simulations for
$r>2$. The absence of wave breaking raises the question of whether
these singular weak solutions would emerge from smooth initial
conditions in finite time. However, wave breaking may not be relevant
to the formation of the $N$-point solutions. In fact, wave breaking
may not be relevant to the creation of CH peakons for $r=2$,
either. See (\cite{bendall2021perspectives}) for a discussion of other
mechanisms of creation of CH peakons. We plan to investigate the
emergence of singular solutions from smooth initial conditions by
solving Equation \eqref{eq:fi} using a finite element method in future
work.  The present work also ignores any analytical questions of local
well-posedness (existence, uniqueness and continuous dependence on
initial conditions) for the r-CH solutions arising from smooth initial
conditions. Finally, it is interesting to ask how the singular solutions
produce solutions of the integrated equation \eqref{eq:fi} even though
they are insufficiently regular to solve the r-CH equation that emerges
from Hamilton's principle. We \response{hypothesise} that this is because the singular
solutions are solutions to the optimal control problem
\begin{equation}
  \min_{u,Q} \int_0^T \|u\|_{\sob{1}{r}}\diff t,
\end{equation}
subject to the constraints
\begin{equation}
  \dot{Q}_i = u(Q_i,t),  \quad
  Q_i(0) = a_i, \, Q_(T) = b_i, \quad i=1,\ldots,N.
\end{equation}
We could then find a minimising sequence of smooth solutions of the
r-CH equation that satisfy these constraints, recovering the
singular solutions in the limit. Since the limit only ensures a solution in $\sob{1}{r}$, we would lose the property of solving
the r-CH equation, but solving the \response{integrated} equation \eqref{eq:fi} is still possible (since all of the smooth
solutions also solve it).
All these questions and crowds of
other related questions that may come easily to mind will be left open
for future work.

\subsection*{Acknowledgements}
We would like to thank our friends and colleagues who have generously
offered their attention, thoughts and encouragement in the course of
this work during the time of COVID-19. We thank Jonathan Mestel for
useful discussions about Equation \eqref{eq:helmholtz}, which
kickstarted this work. CJC is grateful for partial support from EPSRC
(EP/W015439/1, EP/W016125/1, EP/R029423/1, EP/R029628/1, EP/L016613/1)
and NERC (NE/R008795/1). DH is grateful for partial support from ERC
Synergy Grant 856408 - STUOD (Stochastic Transport in Upper Ocean
Dynamics). TP is grateful for partial support
from EPSRC (EP/X017206/1, EP/X030067/1 and EP/W026899/1) and the
Leverhulme Trust (RPG-2021-238).

\appendix

\section{Solving for the characteristic constants $C_i$}
\label{app:Ci}

The sign of $C_i$ determines the characteristic behaviour of the
solution. If $C_i>0$, then $u$ cannot have a zero. Conversely, for
$C_i<0$, $u$ cannot have a turning point. We call these two situations
the ``cosh-like'' and ``sinh-like'' solutions, respectively. When
$C_i=0$, we have exponential solutions $u(x)=a\exp(\pm x)$, with the
constant $a$ and the sign is to be determined from the boundary
conditions: when $x<Q_1$ we have $C_i=0$ and we take the positive
sign. Conversely, we take the negative sign for $x>Q_N$.

To determine $C_i$ for $Q_i<x<Q_{i+1}$, $0<i<N$, given boundary values
$\widehat{u}_i,\widehat{u}_{i+1}$, we first determine the sign of $C_i$. If
$\widehat{u}_i\times \widehat{u}_{i+1}<0$, then by continuity there must be a root,
and so $C_i<0$. If $\widehat{u}_i\times \widehat{u}_{i+1}\geq 0$, we first
eliminate the case $C_i=0$ by fitting the exponential solution with
positive sign if $\widehat{u}_{i+1}>\widehat{u}_i$ and negative sign
otherwise. If the fit is successful, we have determined $C_i=0$. If it
isn't, then either the growth/decay is insufficiently large, and we
have a sinh-like solution with $C_i<0$; or, otherwise, and we have a cosh-like
solution with $C_i>0$.

If $C_i<0$, we determine the sign of $u_x$ from
$\sgn(\widehat{u}_{i+1}-\widehat{u}_i)=\sigma_i$. Then, we integrate
\eqref{eq:ci} to get
\begin{equation}
  \sgn(\widehat{u}_{i+1}-\widehat{u}_i)\int_{\widehat{u}_i}^{u(x)} \frac{\diff w}{(w^r-C_i)^{1/r}}
  =  x - Q_i.
\end{equation}
This implicit equation determines $u(x)$ given $x$, $C_i$, etc. We can
compute the integral using numerical quadrature.

In particular, we have
\begin{equation}
  \sgn(\widehat{u}_{i+1}-\widehat{u}_i)\int_{\widehat{u}_i}^{\widehat{u}_{i+1}}
  \frac{dw}{(w^r-C_i)^{1/r}}
  =  Q_{i+1} - Q_i.
  \label{eq:CuQ_neg}
\end{equation}
This implicit equation relates $C_i$, $Q_{i}$, $Q_{i+1}$, $\widehat{u}_i$
when $C_i<0$.

If $C_i>0$, we have a turning point when $u^r(x^*)=\pm
C_i^{1/r}$. There is only one turning point, because $u$ has no root,
and two turning points would imply a root. If $x^*\in [Q_i,Q_{i+1}]$,
then the sign of $u_X$ changes at the turning point. If $u_i$ and
$u_{i+1}$ are both positive (they need to have the same sign,
otherwise there is a root and $C_i$ must be negative), then $u_x$ must
be negative for $x\in (Q_i,x^*)$ and positive for $x\in
(x^*,Q_{i+1})$. The signs are reversed when $u_i$ and $u_{i+1}$ are
both negative.

Then, we have
\begin{equation}
  \label{eq:CuQ_pos}
  -\sgn(\widehat{u}_i)\int_{\widehat{u}_i}^{C_i^{1/r}}
  \frac{\diff w}{(u^r-C_i)^{1/r}}
  +
  \sgn(\widehat{u}_{i})\int_{C_i^{1/r}}^{\widehat{u}_{i+1}}
    \frac{\diff w}{(u^r-C_i)^{1/r}}
    = Q_{i+1} - Q_i.
\end{equation}
When computing the integrals by numerical quadrature, it is necessary to
remove the weak singularity by a change of variables $v^r=w^r-C_i$,
leading to, e.g.
\begin{equation}
  \int_{u_1}^{C_i^{1/2}} \frac{\diff w}{(w^r-C_i)^{1/r}}
  =  \int_{(u_1^r-C_i)^{1/r}}^{0}
  \frac{v^{r-2}\diff v}{(v^r+C_i)^{\frac{r-1}{r}}},
\end{equation}
for which the integrand would have no singularity at $w^r=C_i$. If $x^*$ is
outside the interval $\Omega_i$ (signalled by equation \eqref{eq:CuQ_pos}
having no solution), then $u_x$ has the same sign throughout the
interval, as determined by the difference $\widehat{u}_{i+1}-\widehat{u}_i$, and
we are back to the situation in \eqref{eq:CuQ_neg}.

\bibliographystyle{elsarticle-harv}
\bibliography{rch}

\end{document}